\documentclass[11pt]{amsart}
\usepackage{amscd,amssymb,palatino}
\usepackage[utf8]{inputenc}
\usepackage[english]{babel}
\usepackage{caption}
\usepackage{etex}
\usepackage[leqno]{amsmath}
\usepackage{amssymb}
\usepackage{color}
\usepackage{shadow}
\usepackage{epsfig}
\usepackage{epic}
\usepackage{graphics}
\usepackage{graphicx}
\usepackage{psfrag}
\usepackage{calc}
\usepackage{tikz}
\usepackage[dvipsnames,prologue,table]{pstricks}
\usepackage{pstcol}
\usetikzlibrary{arrows,decorations,patterns,positioning,automata,shadows,fit,shapes,calc,decorations.markings,backgrounds,scopes,decorations.text}

\usepackage{tipa}

\usepackage{fancyhdr}
\fancypagestyle{plain}{\fancyhf{}
\fancyfoot[C]{{\large\sf\bfseries\thepage}}}
\usepackage{calc}
\newcommand\1{\lower 9pt\hbox{\underbar{}}}
\numberwithin{equation}{section}

\newtheorem {theorem}                   {Theorem}

\newtheorem {lemma}[equation]{Lemma}
\newtheorem {Proposition}[equation]     {Proposition}

\newtheorem {Claim}[equation]      {Claim}

\theoremstyle{definition}
\newtheorem {defi}[equation]{Definition}
\newtheorem {Remark}[equation]          {Remark}

\newcommand{\pr} {\smallskip\noindent{\bf Proof\,\,}}

\newcommand{\pp}[2]{\frac{\partial#1}{\partial#2}}

\begin{document}

\title{On the volume elements of a manifold with transverse zeroes}

\author{Robert Cardona}
\address{{Laboratory of Geometry and Dynamical Systems, Department of Mathematics}, Universitat Polit\`{e}cnica de Catalunya and Barcelona Graduate School of Mathematics BGSMath, EPSEB, Avinguda del Doctor Marañ\'{o}n 44-50, 08028, Barcelona, Spain }
\email{robert.cardona@upc.edu}
\author{Eva Miranda}
\thanks{{ E. Miranda  is supported by the Catalan Institution for Research and Advanced Studies via an ICREA Academia Prize 2016. Robert Cardona acknowledges financial support from the Spanish Ministry of Economy and Competitiveness, through the Mar\'ia de Maeztu Programme for Units of Excellence in R\&D (MDM-2014-0445). Both authors are supported by the grants reference number MTM2015-69135-P (MINECO/FEDER) and reference number 2017SGR932 (AGAUR).  Part of the work that lead to this paper took place at the Fields Institute in Toronto while the second author was invited professor during the Focus Program on Poisson Geometry and Physics in July 2018. This material is based upon work supported by the National Science Foundation under Grant No. DMS-1440140 while the authors were in residence at the Mathematical Sciences Research Institute in Berkeley, California, during the Fall 2018 semester. }}

\address{{Laboratory of Geometry and Dynamical Systems Department of Mathematics}, Universitat Polit\`{e}cnica de Catalunya/Barcelona Graduate School of Mathematics BGSMath, EPSEB, Avinguda del Doctor Marañ\'{o}n 44-50, 08028, Barcelona, Spain \\ and
\\ IMCCE, CNRS-UMR8028, Observatoire de Paris, PSL University, Sorbonne
Universit\'{e}, 77 Avenue Denfert-Rochereau,
75014 Paris, France}
\email{eva.miranda@upc.edu}

\begin{abstract}
Moser proved in 1965  in his seminal paper \cite{moser} that two volume forms on a compact manifold can be conjugated by a diffeomorphism, that is to say they are equivalent, if and only if their associated cohomology classes in the top cohomology group of a manifold coincide. In particular, this yields a classification of compact symplectic surfaces in terms of De Rham cohomology.
  In this paper we generalize these results for volume forms admitting transversal zeroes. In this case there is also a cohomology capturing the classification: the relative cohomology with respect to the critical hypersurface. We compare this classification scheme with the classification  of Poisson structures on surfaces which are symplectic away from a hypersurface where they fulfill a transversality assumption ($b$-Poisson structures). We do this using the desingularization technique introduced in \cite{gmw1} and extend it to $b^m$-Nambu structures.
\end{abstract}
\maketitle

\section{Introduction}

Moser path method is one of the most commonly used  methods in symplectic geometry and topology to prove that two given symplectic structures are equivalent. It first appeared in in Moser's celebrated article \cite{moser} where volume forms on a compact manifold are classified. In particular in dimension 2, a volume form determines a symplectic structure on a surface and Moser's theorem gives a classification of symplectic surfaces. Moser's classification is given in terms of De Rham Cohomology: two forms belong to the same cohomology class if and only if there exists a diffeomorphism conjugating them. Forms conjugated by a diffeomorphism are called \emph{equivalent} for short in this paper.

If we allow the top degree form to have transverse zeroes, asking for the same cohomology class is not enough  to apply Moser's path method. In this case relative cohomology captures the additional information needed.

Following  \cite{God} recall that given
a smooth manifold $M$  and a closed submanifold $Z$, with $i:Z \hookrightarrow M$ the inclusion. The \textbf{relative De Rham cohomology}  groups of $Z$ are given by the complex
$$  \Omega^p(M,Z)= \{ \alpha \in \bigwedge\nolimits^{\!p}T^*M \enspace | \enspace i^*\alpha=0 \}. $$

We will see that in this new scenario additionally having the same relative cohomology allows to apply the Moser's trick.

 Even if the existence of transversal zeroes allows non-orientability in this picture, we will assume our manifolds to be orientable. For the sake of simplicity and mimicking the surface case we will call these volume forms \textbf{folded volume forms}.

In the last part of this paper we study the compatibility between the classification of $b^m$-symplectic surfaces obtained by Geoff Scott in \cite{S} and   our classification scheme. This affinity is studied using the desingularization procedure developed in \cite{gmw1} for $2$-forms. When $m$ is odd, the desingularized structure is a folded-symplectic one. We will see that two equivalent $b^{2k+1}$-symplectic structures are sent to equivalent folded-symplectic forms. We extend this desingularization procedure to volume forms and prove an extension of this result for volume forms.

\hfill \newline
\textbf{Acknowledgements:} We are thankful to Rui Loja Fernandes, Ralph Klaasse, Ioan Marcut and Marco Zambon for useful comments on the first version of this paper.

\section{Preliminaries}

\subsection{Folded singularities and diffeomorphisms of hypersurfaces}

We will be studying top power forms that vanish satisfying a transversal condition\footnote{ This condition can be generalized replacing  standard tranversality by transversality à la Thom.}. Mimicking from the case of $2$-forms \cite{Unf, cannas} we call these structures \textbf{folded volume forms.} As a consequence of transversality, the vanishing set for the top power will always be a closed hypersurface called \textbf{the critical set} and that may have several connected components. In order to have an equivalence relation between these singular forms, the following condition will be imposed on this critical set.

\begin{defi}

Two sets of smooth disjoint oriented hypersurfaces $(S_1,...,S_n)$ and $(S_1',...,S_n')$ are \emph{diffeomorphically equivalent} if there is an orientation-preserving diffeomorphism $\varphi: M \rightarrow M$ mapping the first set to the second one preserving orientations.

\end{defi}

In the space of $n$ disjoint oriented hypersurfaces on a manifold $M$ this condition defines an equivalence relation. Then for a set of $n$ disjoint oriented hypersurfaces $(S_1,...,S_n)$ we denote $[(S_1,...,S_n)]$ its class in the space of diffeomorphically equivalent classes.

\begin{Remark}
When the hypersurfaces are the same we denote by $\operatorname{Diff}(M,Z)$ the set of diffemorphisms  preserving the set of hypersurfaces $Z$.
\end{Remark}

\subsection{A crash course on $b^m$-manifolds}

The category of $b$-manifolds was developed by Melrose \cite{Mel}, in order to study manifolds with boundary. Most of the definitions can be used replacing the boundary by any given hypersurface of the manifold:
\begin{defi}
 A $b$-manifold $(M, Z)$ is an oriented manifold $M$  with an oriented hypersurface
$Z$.
\end{defi}
In order to have the $b$-category we introduce the notion of $b$-map.
\begin{defi}
A $b$-map is a map
$$f : (M_1, Z_1) \longrightarrow (M_2, Z_2)$$
so that $f$ is transverse to $Z_2$ and $f^{-1} (Z_2) = Z_1$.
\end{defi}
Not only maps have to be redefined in the $b$-category, but also vector fields and differential forms:
\begin{defi}
A $b$-vector field on a $b$-manifold $(M,Z)$ is a vector field which is tangent to $Z$ at every point $p\in Z$.
\end{defi}
These vector fields form a Lie subalgebra of vector fields on $M$. Let $t$ be a defining function of $Z$ in a neighborhood $U$ and let $(t,x_2,...,x_n)$ be a chart on it. Then the set of $b$-vector fields on $U$ is a free $C^\infty(U)$-module with basis
$$\left( t \pp{}{t}, \pp{}{x_2},\ldots, \pp{}{x_n}\right).$$
We deduce that the sheaf of $b$-vector fields on $M$ is a locally free $C^\infty$-module and therefore it is given by the sections of a vector bundle on $M$. This vector bundle is called \textbf{the $b$-tangent bundle} and denoted by $^b TM$. Its dual bundle is called \textbf{the $b$-cotangent bundle} and is denoted $^bT^*M$.

By considering sections of powers of this bundle, we can form the so-called \textbf{$b$-forms}.

\begin{defi}
Let $(M^{2n},Z)$ be a $b$-manifold and $\omega\in \,^b \Omega^2(M)$ a closed $b$-form. We say that $\omega$ is $b$-symplectic if $\omega_p$ is of maximal rank as an element of $\Lambda^2(\,^b T_p^* M)$ for all $p\in M$.
\end{defi}

In the class of Poisson manifolds a distinguished subclass is that of $b$-Poisson manifolds which is indeed formed by $b$-symplectic manifolds together with a bi-vector field naturally associated to the $b$-symplectic forms.

\begin{defi}
Let $(M^{2n},\Pi)$ be an oriented Poisson manifold. Let the map
$$p\in M\mapsto(\Pi(p))^n\in \Lambda^{2n}(TM)$$
be transverse to the zero section. Then $\Pi$ is called a $b$-Poisson structure on $M$. The hypersurface $Z$ where the multivectorfield $\Pi^n$ vanishes,
$$Z=\{p\in M|(\Pi(p))^n=0\}$$
 is called the critical hypersurface of $\Pi$. The pair $(M,\Pi)$ is called a $b$-Poisson manifold.
\end{defi}
Asking the transversality condition is equivalent to saying that $0$ is a regular value of the map $p\longrightarrow (\Pi(p))^n$. The hypersurface $Z$ has a defining function obtained by dividing this map by a non-vanishing section of $\bigwedge^{2n}(TM)$.

The set of $b$-symplectic manifolds is  in one-to-one correspondence with the set of $b$-Poisson manifolds.

This correspondence is proved in \cite{GMP2} and can be formulated as
\begin{Proposition}
A two-form $\omega$ on a $b$-manifold $(M,Z)$ is $b$-symplectic if and only if its dual bivector field $\Pi$ is a $b$-Poisson structure.
\end{Proposition}

In this context we have a normal form theorem analogous to Darboux theorem for symplectic manifolds. This results is also proved in \cite{GMP2}.

\begin{theorem}[{\bf $b$-Darboux theorem}]\label{thm:bdarboux}
Let $(M,Z, \omega)$ be a $b$-symplectic manifold. Then, on a neighborhood of a point $p\in Z$, there exist coordinates $(x_1,y_1,...,x_n,y_n)$  centered at $p$ such that
$$\omega=\frac{1}{x_1}\,dx_1\wedge dy_1 + \sum_{i=2}^{n} dx_i\wedge dy_i.$$
\end{theorem}
Note that with this chart, the symplectic foliation of $(M,\Pi)$ has a specific form. It has two open subsets where the Poisson structure has maximal rank given by $\{x_1>0 \}$ and $\{ x_1<0 \}$. The hyperplane $\{x_1=0\}$ contains leaves of dimension $2n-2$ given by the level sets of $y_1$.

One of the research directions has been to generalize $b$-structures and consider more
degenerate singularities of the Poisson structure. This is the case of $b^m$-Poisson structures, for which $\omega^n$ has a singularity of $A_n$-type in Arnold’s list of simple singularities \cite{AN1} \cite{AN2}.
It is convenient, as in the $b$-case, to consider the dual approach and work with forms for their study.

\begin{defi} A symplectic $b^m$-manifold is a pair $(M^{2n}
, Z)$ with a closed $b^m$-two form $\omega$ which
has maximal rank at every $p\in M$.
\end{defi}

Such as in the $b$-symplectic case,  a $b^m$-Darboux theorem holds,
\begin{theorem}[$b
^m$-Darboux theorem, \cite{gmw1}]
    Let $\omega$ be a $b^m$-symplectic form on $(M^{2n}
, Z)$ and $p\in
Z$. Then we can find a coordinate chart $(x_1,y_1,...,x_n,y_n)$ centered at $p$ such that the hypersurface $Z$ is locally defined by $\{y_1=0\}$ and
$$ \omega = dx_1 \wedge  \tfrac{dy_1}{y_1^m} + \sum_{i=2}^n dx_i \wedge dy_i .$$
\end{theorem}

Dualizing we obtain the Darboux form for the $b^m$-Poisson bivector field,
$$ \Pi= y_1^m \pp{}{x_1} \wedge \pp{}{y_1} + \sum_{i=2}^n \pp{}{x_1} \wedge \pp{}{y_1} .$$

A decomposition for these forms is given in \cite{S}.

\begin{defi} A Laurent Series of a closed $b^m$-form $\omega$ is a decomposition of $\omega$ in a tubular
neighborhood $U$ of $Z$ of the form
	$$ \omega= \frac{dx}{x^m}\wedge ( \sum_{i=0}^{m-1} \pi^*(\hat{\alpha}_i)x^i)+\beta, $$
where $\pi:U\rightarrow Z$ is the projection, where each $\hat{\alpha}_i$ is a closed form on $Z$, and $\beta$ is form on U.
\end{defi}
And there is a result concerning this decomposition of $\omega$.
\begin{Proposition}
 In a tubular neighborhood of $Z$, every closed $b^m$-form $\omega$ can be written
in a Laurent form and the restriction of $\sum_{i=0}^{m-1} \pi^*(\hat{\alpha}_i)x^i$ and $\beta$ to $Z$ are well-defined closed 1 and 2-forms respectively.
\end{Proposition}

\section{A Moser trick for transversally vanishing volume forms}

In order to apply the Moser's path method in this case, we need to prove a few auxiliary lemmas. Let $\Omega$ be a transversally vanishing volume form with critical set $\bar Z$. In what follows we will denote $Z$ any of the connected components of the critical set and denote by $t$ a defining function of it.

 Observe that given a top degree form $\mu$ on $U$,  a neighborhood of $Z$, the form $t\mu$ is a transversally vanishing volume form (in a possibly smaller neighborhood) having $Z$ as critical set if and only if $\mu$ is non-vanishing along $Z$.

 Let  $\Omega_0$ and $\Omega_1$ stand for two transversally vanishing volume forms  at $\bar Z$ which for simplicity will be denoted as \textbf{folded volume forms}. In what follows we assume that the orientation induced on each component of $\bar Z$ is the same for both forms.
\begin{lemma}\label{path}
For $0\leq s \leq 1$, the form
$$ \Omega_s= (1-s)\Omega_0 + s\Omega_1 $$
is a folded volume form having $Z$ as critical set.

\end{lemma}
\begin{proof}
By the argument described above we may write $\Omega_0= t\mu_0$ and $\Omega_1=t\mu_1$ for $\mu_0$ and $\mu_1$ not vanishing at $Z$ and positive  (because of matching orientations). Consider the path $\mu_s= (1-s)\mu_0 + s\mu_1$ for $0\leq s \leq 1$. Observe that $\Omega_s=t\mu_s$ and thus $\mu_s$ does not vanish at $Z$.
\end{proof}

A consequence is that $\iota_v\Omega_s$ vanishes along $Z$, where $v$ is any non-vanishing section of $TM$ (or $TU$). By this lemma we deduce,
\begin{Claim}\label{vf}
Given $\alpha \in \Omega^{n-1}(U)$, there exists a vector field $u$ such that
$$ \iota_u\Omega_s= \alpha $$
if and only if $ \alpha|_Z=0$.
\end{Claim}
Observe that since in $M \backslash \bar Z$ the form defines a volume, if the vector field exists it is unique.

Assume now that both the usual and relative cohomology class with respect to $Z$ of $\Omega_0$ and $\Omega_1$ coincide. Then there is $\beta$ such that $\Omega_0-\Omega_1=d\beta$. By definition we have that $i^*\beta=0$, where $i:Z \hookrightarrow M$ is the inclusion of $Z$ in $M$.

\begin{lemma}\label{wein}
We can assume that $\beta$ satisfies $\beta|_Z=0$.
\end{lemma}

\begin{proof}
For this we need to recall the relative Poincar\'e lemma for which we follow \cite{W}.
\begin{theorem}[Relative Poincar\'e lemma]
Let $N \subset M$ be a closed submanifold of $M$, and $\omega$ a closed $k$-form of $M$ whose pullback to $N$ is zero. Then there is a $(p-1)$-form $\lambda$ on a neighborhood of $N$ such that $d\lambda=\omega$ and $\lambda$ satisfies $i^*\lambda=0$. If $\omega$ satisfies $\omega|_N=0$ then $\lambda$ can be chosen such that $\lambda|_N=0$.
\end{theorem}
Since the relative cohomology vanishes, we have $\beta$ such that $i^*\beta=0$. In a neighborhood $U(Z)$ of $Z$, we can apply the relative Poincar\'e lemma and there exist a $1$-form $\lambda$ in this neighborhood such that $\Omega_0-\Omega_1= d\lambda$ and $\lambda|_Z=0$. In this neighborhood $d\beta = d\lambda$ and $i^*(\beta-\lambda)=0$ so  the relative Poincar\'e lemma yields the existence of a form $\alpha$ such that $\beta-\lambda= d\alpha$. Observe that in $Z$ we have $d\alpha|_Z= \beta|_Z$.

Let $\varphi$ be a bump function of a possibly smaller neighborhood of $Z$ and consider $\varphi\alpha$ a global extension of $\alpha$ to $M$. Then the form $ \gamma= \beta - d(\varphi \alpha)$ satisfies $\gamma|_Z=0$ and $\Omega_0-\Omega_1= d\gamma$. This completes the proof of the lemma.

\end{proof}

We can improve this statement by having a more explicit expression for $\beta$. This will give some information about the isomorphism that we obtain via Moser's trick.

\begin{lemma}\label{rela}
 The form $\beta$ can be  written as $\beta|_U= t^2 \alpha$ in a neighborhood of each connected component of $\bar Z$.

\end{lemma}

\begin{proof}

The fact the the relative cohomology of $\Omega_0-\Omega_1$ is zero means that we can assume that $\beta$ vanishes at $TM|_z$ for every point $z\in Z$ because of the previous lemma. In particular in a possibly smaller neighborhood $U$ it is of the form $\beta= t\alpha$ for an $\alpha \in \Omega^{n-1}(U)$. Observe that $d\beta= dt\wedge \alpha - td\alpha$ but also $d\beta= \Omega_0-\Omega_1= t\mu$. Thus $\alpha$ needs to vanish at least linearly at $Z$; in particular $\beta$ vanishes at least at order $2$ in $t$.

\end{proof}

We can now state and prove a version of Moser's theorem for transversally vanishing volume forms.

\begin{theorem}\label{main}
Let $\Omega_0$ and $\Omega_1$ be two folded volume forms with critical set $\bar Z=Z_1 \cup ... \cup Z_n$. Assume that  the cohomology classes of $\Omega_0$ and $\Omega_1$ coincide in both De Rham cohomology and relative cohomology (i.e., $[\Omega_0]= [\Omega_1]$ and $[\Omega_0]_r=[\Omega_1]_r$), then there exist a diffeomorphism $\varphi$ such that $\varphi^*\Omega_1=\Omega_0$ that restricts to the identity along  $\bar Z$.

\end{theorem}

\begin{proof}

Since the De Rham cohomology class of $\Omega_0$ is the same as $\Omega_1$, the following equality holds $\Omega_0-\Omega_1=d\beta$.

Let $Z$  be one of the connected components of $\bar Z$ and  let $v$ be an oriented non-vanishing section of $TM$. Denoting by $U=U(Z)$, a neighborhood of $Z$, we may write $\Omega_i|_U=t\mu_i$ with $\mu_i$ is a non-vanishing form and $t$ a defining function of $Z$, for $i=1,2$.

Consider now  the path  $\Omega_s= (1-s) \Omega_0 + s \Omega_1$ for $s\in [0,1]$. By  Lemma \ref{path}, $\Omega_s$ is vanishing transversally at the same critical set thus $\Omega_s|_Z=0$. Because the relative cohomology class at $Z$ of the two forms is the same, in a possibly smaller neighborhood we may apply Lemmas \ref{wein} and \ref{rela} and around $Z$ the form is written as  $\beta = t^2 \alpha$ with $t$ a defining function of $Z$.
The same applies for any of the connected components in $\bar Z$.  In order to apply Moser's trick we need to solve the equation
$$ \mathcal{L}_{v_s}\Omega_s + \frac{d\Omega_s}{ds}=0, $$
\noindent which may  be written as $d\iota_{v_s}\Omega_s= \Omega_0-\Omega_1= d\beta$. This is equivalent to finding a vector field $v_s$ satisfying
$$ \iota_{v_s}\Omega_s= \beta. $$
Because Lemma \ref{vf} applies for any curve in $\bar Z$, there exist a unique solution to the equation. Now since $\beta$ vanishes to second order, $v_s$ vanishes to the first order in all the components of the critical set. The flow $\varphi_s$ of $v_s$ satisfies $\varphi_s^*\Omega_s=\Omega_0$, hence $\varphi_1$ is the desired diffeomorphism.  Observe that this  diffeomorphism restricts to identity in the critical set.
\end{proof}

The theorem also applies if the critical sets of $\Omega_0$ and $\Omega_1$ are diffeomorphically equivalent by an orientation-preserving diffeomorphism. The fact that the relative cohomology is invariant for equivalent  folded volume forms needs an extra assumption in the general setting.

\begin{theorem} Let $\varphi$ be a diffeomorphism in the arc-connected component of the identity in $\operatorname{Diff}(M, Z)$  and $\Omega_0$ and $\Omega_1$ two folded volume forms such that $\varphi^*\Omega_1=\Omega_0$ then the cohomology classes determined by  $\Omega_0$ and $\Omega_1$ are the same in De Rham cohomology and in relative cohomology  (i.e., $[\Omega_0]= [\Omega_1]$ and $[\Omega_0]_r=[\Omega_1]_r$).
\end{theorem}

\begin{proof}

Since $\varphi$ belongs to the arc-connected component of the
identity, we can indeed construct an homotopy $\varphi_t$ leaving $Z$ invariant such that
$\varphi_1=\varphi$ and $\varphi_0=id$. Denote $\Omega= \Omega_1-\Omega_0$.

 We can use this homotopy to
define a de Rham homotopy operator:

$$Q\Omega=\int_0^1\varphi_t^*(\iota_{v_t}\Omega) dt$$

\noindent where $v_t$ is the t-dependent vector field defined by the
isotopy $\varphi_t$.

Using this formula, we can prove (see for instance pages 110 and 111
in \cite{GS})
that $[\Omega_1]=[\Omega_0]$ as we can write $\Omega_1=\Omega_0+d\alpha$
 for the $1$-form  $\alpha=Q\Omega$.
From the formula above we can check that the relative cohomology class is also the same. Since $\Omega$ vanishes at $Z$, we deduce that $Q\Omega$ also vanishes at $Z$ and in particular its pullback to $Z$ is zero.

\end{proof}

\section{ Compatibiity of the classification of $b^m$-structures and the desingularization transformation}

\subsection{Desingularizing $b^m$-forms}

In \cite{gmw1} the desingularization of $b^{m}$-forms was introduced, leading to a radical new approach to the study of obstruction theory for the existence  of $b^m$-symplectic structure on a prescribed manifold.

 We will now detail how this desingularization can be applied to any $b^m$-form of any degree. This idea was already applied to $1$-forms for the study of singular contact structures in \cite{mo}.

Let $(M,Z)$ be a $b^m$-manifold and $\omega$ a $b^m$-form of degree $p$. Denote  by $t$ a defining function for $Z$. Following section $3.2$ in \cite{S}, $\omega$ can be written in a neighborhood $U$ of $Z$ as,
$$ \omega = \frac{dt}{t^m}\wedge \alpha + \beta,  $$
for $\alpha \in \Omega^{p-1}(U)$ and $\beta \in \Omega^p(U)$ where $U$ is an $\epsilon$-neighborhood of $Z$.  This decomposition is not unique as observed already in \cite{GMP2} and \cite{S}.

 In what follows , we consider as fixed the decomposition. As in \cite{gmw1} two different cases have to be considered depending on the parity of $m$.

\vspace{5mm}
\textbf{Case I: even $m$ .}

\vspace{5mm}

Assume $m=2k$ and let $f\in \mathcal{C}^{\infty}(\mathbb R)$ be an odd smooth function satisfying $f'(x)>0$ for all $x \in [-1,1]$ as shown below,

\begin{figure}[h!]
\centering
\includegraphics[scale=.6]{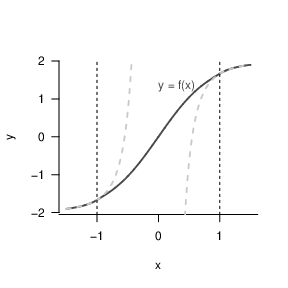}
\label{fig:digraph}
\end{figure}

 and satisfying

\[f(x)=\begin{cases} \frac{-1}{(2k-1) x^{2k-1}}-2 &\textrm{for} \quad x<-1  \\ \frac{-1}{(2k-1) x^{2k-1}}+2 &\textrm{for} \quad x>1
\end{cases}\]

\noindent outside the interval $[-1,1]$.

Scaling the function consider the function
 \begin{equation*}\label{definingequation}f_\epsilon(x):=  \frac{1}{\epsilon^{2k-1}}
f \left(\frac{x}{\epsilon}\right).
\end{equation*}
And outside the interval,
\[f_\epsilon(x)=\begin{cases} \frac{-1}{(2k-1) x^{2k-1}}-\frac{2}{\epsilon^{2k-1}} &\textrm{for} \quad x<-\epsilon  \\ \frac{-1}{
(2k-1)x^{2k-1}}+\frac{2}{\epsilon^{2k-1}}&\textrm{for} \quad x>\epsilon
\end{cases}\]
 Replacing $\frac{dx}{x^{2k}}$ by $df_\epsilon$  in the semi-local expression on $U$ and obtain
 $$  \omega_\epsilon = df_\epsilon \wedge \alpha + \beta. $$
We call it a $f_\epsilon$-desingularization of $\omega$.

\vspace{5mm}
\textbf{Case II: odd $m$.}
\vspace{5mm}

Consider $m=2k+1$, and consider a function $f\in C^\infty(\mathbb{R})$ satisfying
\begin{itemize}

\item $f(x)=f(-x)$
\item  $f'(x)>0$ if $x>0$
\item $f(x)=x^2-2$ if $x\in [-1,1]$
\item $f(x)=\log(\vert x \vert)$ if $k=0$, $x\in \mathbb R\setminus[-2,2]$
\item $f(x)=-\frac{1}{(2k+2)x^{2k+2}}$ if $k>0$, $x\in \mathbb R\setminus[-2,2]$.
\end{itemize}
\begin{center}
\includegraphics[scale=0.55]{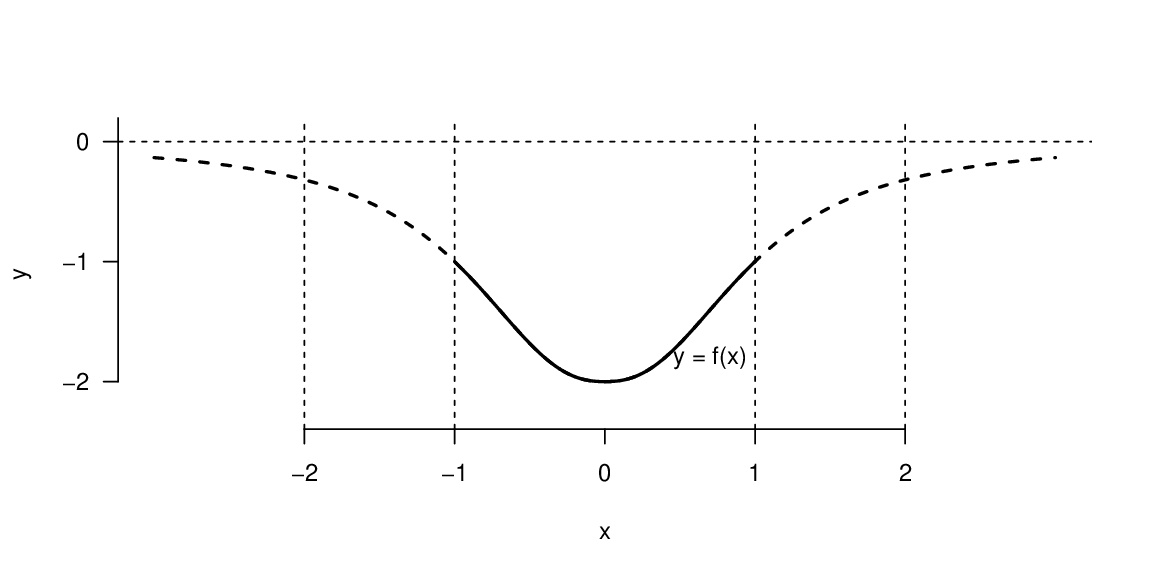}
\end{center}
Taking $\epsilon$  the width of a tubular neighborhood of $Z$ define
\begin{equation*}
f_\epsilon(x):=  \frac{1}{\epsilon^{2k}} f\left(\frac{x}{\epsilon}\right)
\end{equation*}
and consider the form

$$ \omega_\epsilon = df_\epsilon \wedge \alpha + \beta. $$
Observe that the $f_\epsilon$-desingularization is again smooth and  $df_\epsilon$ vanishes transversally at $Z$.

\begin{Remark}
When $\omega$ is closed, its Laurent decomposition can be used as in \cite{gmw1} to prove that $\omega_\epsilon$ is also closed.

\end{Remark}

\subsection{Compatibility of the different classification schemes}

The aim of this section is to relate the classification of $b^{2k+1}$-symplectic surfaces and the theorem proved in section 3, using the desingularization formulas described above. Recall that $b$-symplectic structures were classified by O.Radko in \cite{radko}. In \cite{radko} Radko uses  the notion of diffeomorphism class of curves and uses cohomology and together with the modular period to classify \emph{stable Poisson structures on surfaces}. Later on Scott classifies $b^m$-structures in surfaces (see theorem 6.7 in \cite{S}).

\begin{theorem}[Scott, Classification of $b^m$-surfaces]\label{scott}
Let $\omega_0, \omega_1$ be two symplectic $b^m$-forms on a compact connected $b^m$-surface $(M,Z,j_Z)$. The following are equivalent
\begin{enumerate}
\item The forms $\omega_0, \omega_1$ are $b^m$-symplectomorphic.
\item Their $b^m$-cohomology class is equal $[\omega_0] = [\omega_1]$.
\item The Liouville volumes of $\omega_0$ and $\omega_1$ agree, as do the numbers
 $$  \int_{\gamma_r} \alpha_{-i}, $$
 for all connected components $\gamma_r \subset Z$ and all $1\leq i \leq k$, where $\alpha_{-i}$ are the terms appearing in the Laurent decomposition of the two forms.
\end{enumerate}
\end{theorem}

We can also consider top degree volume forms in $b^m$-manifolds as studied in \cite{mp1}, \cite{mp2} and introduced in \cite{N}. These forms, called $b^m$-Nambu forms, satisfy also that if two of them have the same $b^m$-cohomology then they are isomorphic.

\begin{theorem}
Let $\Theta_0$ and $\Theta_1$ be two $b^m$-Nambu structures of degree $n$ on a compact orientable manifold $M^n$. If $[\Theta_0]=[\Theta_1]$ in $b^m$-cohomology then there exists a diffeomorphism $\varphi$ such that $\varphi^*\Theta_1=\Theta_0$.
\end{theorem}
\begin{Remark}
In fact two of these $b^m$-Nambu structures are equivalent if and only if their $b^m$-cohomology classes coincide. This can be proved as it is done for surfaces in \cite{S} and the theorem can be stated as Theorem \ref{scott} replacing $b^m$-symplectic forms by $b^m$-Nambu structures. Since the $b^m$-Nambu structures of top degree are closed $b^m$-forms they admit a Laurent decomposition. It is detailed in section 5 of \cite{S} where  the class in $b^m$-cohomology $[\Theta]$ is identified with its Liouville-Laurent decomposition $([\Theta_{sm}],[\alpha_1],...,[\alpha_m])$. This is in fact the $b^m$-Mazzeo-Melrose isomorphism for the top degree
$$ ^{b^m}H^n(M) \cong H^n(M) \oplus (H^{n-1}(Z))^m .$$
Using $\alpha_1,...,\alpha_m$ the modular periods of $[\Theta]$ associated to each modular $(n-1)$-form can be determined and it can be proved that they are invariant as it is done in \cite{Mar} for $b$-Nambu structures.
\end{Remark}

We can now state a compatibility theorem between this classification and its desingularized form.

\begin{theorem}
Let $\Theta_0$ and $\Theta_1$ be two $b^{2k+1}$-Nambu structures in a $b^{2k+1}$-manifold that are equivalent then for all $\epsilon$ the $f_\epsilon$-desingularized forms are also equivalent as folded volume forms (i.e., there exists a diffeomorphism conjugating them).
\end{theorem}

\begin{proof}
Since the forms are equivalent the classes  satisfy  $[\Theta_0] = [\Theta_1]$ in $b^{2k+1}$-cohomology.

Denote $t$ a defining function of $Z$. The forms $\Theta_0$ and $\Theta_1$ can be written close to any connected component of $\bar Z$ as:
$$ \Theta_j = \alpha_j' + \beta_j \wedge \frac{dt}{t^{2k+1}},  $$
where $t$ is a defining function of the component of $\bar Z$. Since $\alpha_j'$ is a $n$ form in $M$ it can be written as $\alpha_j' = \gamma_j \wedge dt = \gamma_j t^{2k+1} \wedge \frac{dt}{t^{2k+1}}$. Hence denoting as $\alpha_j := \gamma_j t^{2k+1} + \beta_j$, as in section 6.4 of \cite{GMP2}, the forms can be decomposed as
 $$\Theta_j= \alpha_j \wedge \frac{dt}{t^{2k+1}}.$$

Then $\Theta_1-\Theta_0 = (\alpha_0 - \alpha_1) \wedge \frac{dt}{t^{2k+1}}= d\mu \wedge \frac{dt}{t^{2k+1}}$ because they have the same $b^{2k+1}$-cohomology class. Once applying the desingularizing procedure, we obtain,
$$ \Theta_{1,\epsilon} - \Theta_{0,\epsilon} = d\mu \wedge df_{\epsilon}, $$
and the right  hand side looks locally as $\frac{2}{\epsilon^{2k+2}}t d\mu \wedge dt = d(\frac{2}{\epsilon^{2k+2}}t \mu \wedge dt) $.

\hfill \newline
We deduce that for any $\epsilon$ the forms $\Theta_{1,\epsilon}$ and $\Theta_{0,\epsilon}$ have the same cohomology class in $H^n(M)$ and same relative cohomology class in $H^n(M,Z)$, because they are exact with respect to a form $\beta=\frac{2}{\epsilon^{2k+2}}t \mu \wedge dt$ that vanishes at $Z$. Applying Theorem \ref{main} we deduce that these two forms are isomorphic as folded volume forms.
\end{proof}

 As a remark, observe that the desingularized forms we consider depend on the decomposition in use. We obtain a compatibility theorem for the classification of $b^{2k+1}$-Nambu structures. Thus equivalent $b^{2k+1}$-Nambu structures are sent to equivalent folded volume forms. When the dimension of the manifold is $2$, the compatibility is hence between $b^{2k+1}$-symplectic forms and folded symplectic forms.

\end{document}